\documentclass[twoside,reqno]{amsart}
\usepackage{amsfonts}
\usepackage{graphicx}
\usepackage{amscd}
\usepackage{cite}
\usepackage{hyperref}

\newtheorem{theorem}{Theorem}
\theoremstyle{plain}

\newtheorem{corollary}{Corollary}

\newtheorem{lemma}{Lemma}

\newtheorem{remark}{Remark}

\numberwithin{equation}{section}

\input{tcilatex}

\begin{document}
\title[An Implicit Algorithm for Multivalued Mappings]{A New Iterative
Projection Method for Approximating Fixed Point Problems and Variational
Inequality Problems}
\author{Ibrahim Karahan}
\address{Department of Mathematics, Erzurum Technical University, Erzurum
25240, Turkey}
\email{ibrahimkarahan@erzurum.edu.tr}
\author{Murat Ozdemir}
\address{Department of Mathematics, Ataturk University, Erzurum 25240, Turkey%
}
\email{mozdemir@atauni.edu.tr}

\begin{abstract}
In this paper, we introduce and study a new extragradient iterative process
for finding a common element of the set of fixed points of an infinite
family of nonexpansive mappings and the set of solutions of a variational
inequality for an inverse strongly monotone mapping in a real Hilbert space.
Also, we prove that under quite mild conditions the iterative sequence
defined by our new extragradient method converges strongly to a solution of
the fixed point problem for an infinite family of nonexpansive mappings and
the classical variational inequality problem. In addition, utilizing this
result, we provide some applications of the considered problem not just
giving a pure extension of existing mathematical problems.
\end{abstract}

\maketitle

\section{Introduction}

Throughout this paper, we assume that $H$ is a real Hilbert space whose
inner product and norm are denoted by $\left\langle \cdot ,\cdot
\right\rangle $ and $\left\Vert \cdot \right\Vert $, respectively, $C$ is a
nonempty closed convex subset of $H$, and $I$ is the idendity mapping on $C$%
. Below, we gather some basic definitions and results which are needed in
the subsequent sections. Recall that a mapping $T:C\rightarrow C$ is called
nonexpansive if 
\begin{equation*}
\left\Vert Tx-Ty\right\Vert \leq \left\Vert x-y\right\Vert ,\text{ }\forall
x,y\in C.
\end{equation*}%
We denote by $F(T)$ the set of fixed points of $T$. For a mapping $%
A:C\rightarrow H,$ it is said to be

\begin{enumerate}
\item[(i)] monotone if%
\begin{equation*}
\left\langle Ax-Ay,x-y\right\rangle \geq 0,\text{ }\forall x,y\in C;
\end{equation*}

\item[(ii)] $L$-Lipschitzian if there exists a constant $L>0$ such that 
\begin{equation*}
\left\Vert Ax-Ay\right\Vert \leq L\left\Vert x-y\right\Vert ,\text{ }\forall
x,y\in C;
\end{equation*}

\item[(iii)] $\alpha $-inverse strongly monotone if there exists a positive
real number $\alpha >0$ such that%
\begin{equation*}
\left\langle Ax-Ay,x-y\right\rangle \geq \alpha \left\Vert Ax-Ay\right\Vert
^{2},\text{ }\forall x,y\in C.
\end{equation*}
\end{enumerate}

\begin{remark}
\label{z}It is obvious that any $\alpha $-inverse strongly monotone mapping $%
A$ is monotone and $\frac{1}{\alpha }$-Lipschitz continuous.
\end{remark}

\begin{remark}
\label{y}Every $L$-Lipschitzan mapping is $\frac{2}{L}$-inverse strongly
monotone mapping.
\end{remark}

For a mapping $A:C\rightarrow H$, the classical variational inequality
problem $VI\left( C,A\right) $ is to find a $x\in C$ such that%
\begin{equation}
\left\langle Ax,y-x\right\rangle \geq 0,\text{ }\forall y\in C,  \label{25}
\end{equation}%
which is the optimality condition for the minimization problem%
\begin{equation}
\min_{x\in C}\frac{1}{2}\left\langle Ax,x\right\rangle .  \label{13}
\end{equation}%
The set of solutions of $VI\left( C,A\right) $ is denoted by $\Omega ,$
i.e., 
\begin{equation*}
\Omega =\left\{ x\in C:\text{ }\left\langle Ax,y-x\right\rangle \geq 0\text{%
, }\forall y\in C\right\} .
\end{equation*}

In the context of the variational inequality problem it is easy to check that%
\begin{equation*}
x\in \Omega \Leftrightarrow x\in F\left( P_{C}\left( I-\lambda A\right)
\right) ,\text{ }\forall \lambda >0.
\end{equation*}%
Variational inequalities were initially studied by Stampacchia \cite{stam1}, 
\cite{stam2}. Such a problem is connected with convex minimization problem,
the complementarity problem, the problem of finding point $x\in C$
satisfying $0\in Ax$ and etc.. Fixed point problems are also closely related
to the variational inequality problems. Based on this relationship,
iterative methods for nonexpansive mappings have recently been applied to
find the common solution of fixed point problems and variational inequality
problems; see, for example \cite{yalis, caku, lita2, wosaya, yaliya} and the
references therein. Below, we give some of them.

In 2005, Iiduka and Takahashi \cite{lita} proposed an iterative process as
follows:%
\begin{equation}
\left\{ 
\begin{array}{l}
x_{1}\in C, \\ 
x_{n+1}=\alpha _{n}x+\left( 1-\alpha _{n}\right) TP_{C}\left( I-\lambda
_{n}A\right) x_{n},\text{ }\forall n\geq 1\text{,}%
\end{array}%
\right.  \label{D}
\end{equation}%
where $A$ is an $\alpha $-inverse strongly monotone mapping, $\left\{ \alpha
_{n}\right\} \subset \left( 0,1\right) $ and $\left\{ \lambda _{n}\right\}
\in \left( 0,2\alpha \right) $ satisfy some parameters controlling
conditions. They showed that if $F\left( T\right) \cap \Omega $\ is
nonempty, then the sequence $\left\{ x_{n}\right\} $ generated by (\ref{D})
converges strongly to some $z\in F\left( T\right) \cap \Omega $.

One year later, in 2006, by a narrow margin from the iterative process (\ref%
{D}), Takahashi and Toyoda \cite{tato} introduced the following iterative
process which is based on the Mann iteration \cite{mann}:%
\begin{equation}
\left\{ 
\begin{array}{l}
x_{0}\in C, \\ 
x_{n+1}=\alpha _{n}x_{n}+\left( 1-\alpha _{n}\right) TP_{C}\left( I-\lambda
_{n}A\right) x_{n},\text{ }\forall n\geq 0,%
\end{array}%
\right.   \label{A}
\end{equation}%
where $C$ is a nonempty closed convex subset of a real Hilbert space $H,$ $%
P_{C}:H\rightarrow C$ is a metric projection, $A:C\rightarrow H$ is an $%
\alpha $-inverse strongly monotone mapping, and $T:C\rightarrow C$ is a
nonexpansive mapping. They proved that if the set of fixed points of $T$ is
nonempty, then the sequence $\left\{ x_{n}\right\} $ generated by (\ref{A})
converges weakly to some $z\in F\left( T\right) \cap \Omega $ where $%
z=\lim_{n\rightarrow \infty }P_{F\left( T\right) \cap \Omega }x_{n}.$ In the
same year, Yao et. al. \cite{yaliya}, introduced following iterative scheme
for a nonexpansive mapping $S,$ and a monotone $k$-Lipschitzian continuous
mapping $A$. Under the suitable conditions, they proved the strong
convergence of $\left\{ x_{n}\right\} $ for a fixed $u\in H$\ and a given $%
x_{0}\in H$ arbitrary. 
\begin{equation}
\left\{ 
\begin{array}{l}
x_{n+1}=\alpha _{n}u+\beta _{n}x_{n}+\gamma _{n}SP_{C}\left( x_{n}-\lambda
_{n}y_{n}\right) , \\ 
y_{n}=P_{C}\left( I-\lambda _{n}A\right) x_{n},\text{ }\forall n\geq 0.%
\end{array}%
\right.   \label{B}
\end{equation}

Lastly, Khan \cite{khan} and Sahu \cite{sahu}, individually, introduced the
following iterative process which Khan referred to as Picard-Mann hybrid
iterative process:%
\begin{equation}
\left\{ 
\begin{array}{l}
x_{1}\in C, \\ 
x_{n+1}=Ty_{n}, \\ 
y_{n}=\alpha _{n}x_{n}+\left( 1-\alpha _{n}\right) Tx_{n},\text{ }\forall
n\geq 1,%
\end{array}%
\right.  \label{C}
\end{equation}%
where $\left\{ \alpha _{n}\right\} $ is a sequence in $\left( 0,1\right) .$
Picard-Mann hybrid iterative process is independent of all Picard, Mann and
Ishikawa iterative processes.\ Khan \cite{khan} showed that the process (\ref%
{C}) converges faster than all of Picard, Mann and Ishikawa iterative
processes for contractions. Moreover, he proved a strong and a weak
convergence theorems in Banach space for iterative process (\ref{C}) with a $%
T$ nonexpansive mapping under the suitable conditions.

In addition to all these studies, the existence of common elements of the
set of common fixed points of an infinite family of nonlinear mappings and
the set of solutions of the variational inequality problem has been also
considered by many authors (see \cite{yao1, rabian, wang1}). In such
articles, authors usually use a mapping generated by nonexpansive mappings
such as the mapping $W_{n}$ defined, as in Shimoji and Takahashi \cite{shi},
by%
\begin{eqnarray}
U_{n,n+1} &=&I  \notag \\
U_{n,n} &=&\mu _{n}T_{n}U_{n,n+1}+\left( 1-\mu _{n}\right) I  \notag \\
U_{n,n-1} &=&\mu _{n-1}T_{n-1}U_{n,n}+\left( 1-\mu _{n-1}\right) I  \notag \\
&&\vdots   \notag \\
U_{n,k+1} &=&\mu _{k+1}T_{k+1}U_{n,k+2}+\left( 1-\mu _{k+1}\right) I
\label{11} \\
U_{n,k} &=&\mu _{k}T_{k}U_{n,k+1}+\left( 1-\mu _{k}\right) I  \notag \\
&&\vdots   \notag \\
U_{n,2} &=&\mu _{2}T_{2}U_{n,3}+\left( 1-\mu _{2}\right) I  \notag \\
W_{n} &=&U_{n,1}=\mu _{1}T_{1}U_{n,2}+\left( 1-\mu _{1}\right) I  \notag
\end{eqnarray}%
where $C$ is a nonempty closed convex subset of a Hilbert space $H,$ $\mu
_{1},\mu _{2},\ldots $ are real numbers such that $0\leq \mu _{n}\leq 1,$
and $T_{1},T_{2},\ldots $ is an infinite family of self-mappings on $C$. $%
W_{n}$ is called $W$-mapping generated by $T_{n},T_{n-1},\ldots ,T_{1}$ and $%
\mu _{n},\mu _{n-1},\ldots ,\mu _{1}.$ It is clear that nonexpansivity of
each $T_{i}$, $i\geq 1$, ensures the nonexpansivity of $W_{n}$.

In this paper, motivated and inspired by the above processes and
independently from all of them, we introduce the following iterative process
for an infinite family of nonexpansive mappings $\left\{ T_{n}\right\} $
which is based on Picard-Mann hybrid iterative process:%
\begin{equation}
\left\{ 
\begin{array}{l}
x_{0}\in C \\ 
x_{n+1}=W_{n}P_{C}\left( I-\lambda _{n}A\right) y_{n} \\ 
y_{n}=\left( 1-\alpha _{n}\right) x_{n}+\alpha _{n}W_{n}P_{C}\left(
I-\lambda _{n}A\right) x_{n},\text{ }\forall n\geq 0,%
\end{array}%
\right.   \label{12}
\end{equation}%
where $A:C\rightarrow H$ is an $\alpha $-inverse strongly monotone mapping, $%
W_{n}$ is a mapping defined by (\ref{11}), $\{\lambda _{n}\}\subset \lbrack
a,b]$ for some $a,b\in (0,2\alpha )$ and $\left\{ \alpha _{n}\right\}
\subset \left[ c,d\right] $ for some $c,d\in \left( 0,1\right) $. Also, we
prove that the sequence $\left\{ x_{n}\right\} $ defined by (\ref{12})
converge strongly to a common element of the set of common fixed points of
the infinite family $\left\{ T_{n}\right\} $ and the set of solutions of the
variational inequality (\ref{25}) which is the optimality condition for the
minimization problem (\ref{13}).

\section{Preliminaries}

In this section, we collect some useful lemmas that will be used for our
main result in the next section. We write $x_{n}\rightharpoonup x$ to
indicate that the sequence $\left\{ x_{n}\right\} $ converges weakly to $x,$
and $x_{n}\rightarrow x$ for the strong convergence.

It is well known that for any $x\in H,$ there exists a unique point $%
y_{0}\in C$ such that%
\begin{equation*}
\left\Vert x-y_{0}\right\Vert =\inf \left\{ \left\Vert x-y\right\Vert :y\in
C\right\} .
\end{equation*}%
We denote $y_{0}$ by $P_{C}x,$ where $P_{C}$ is called the metric projection
of $H$ onto $C.$ We know that $P_{C}$ is a nonexpansive mapping. It is also
known that $P_{C}$ has the following properties:

\begin{enumerate}
\item[(i)] $\left\Vert P_{C}x-P_{C}y\right\Vert \leq \left\Vert
x-y\right\Vert ,$ for all $x,y\in H,$

\item[(ii)] $\left\Vert x-y\right\Vert ^{2}\geq \left\Vert
x-P_{C}x\right\Vert ^{2}+\left\Vert y-P_{C}x\right\Vert ^{2},$ for all $x\in
H,$ $y\in C,$

\item[(iii)] $\left\langle x-P_{C}x,y-P_{C}x\right\rangle \leq 0,$ for all $%
x\in H,$ $y\in C.$
\end{enumerate}

It is known that a Hilbert space $H$ satisfies the Opial condition that, for
any sequence $\left\{ x_{n}\right\} $ with $x_{n}\rightharpoonup x,$ the
inequality%
\begin{equation*}
\lim \inf_{n\rightarrow \infty }\left\Vert x_{n}-x\right\Vert <\lim
\inf_{n\rightarrow \infty }\left\Vert x_{n}-y\right\Vert
\end{equation*}%
holds for every $y\in H$ with $y\neq x.$

\begin{lemma}
\label{c}\cite{tato} Let $C$ be a nonempty closed convex subset of a real
Hilbert space $H$ and $\left\{ x_{n}\right\} $ be a sequence in $H.$ Suppose
that, for all $z\in C,$%
\begin{equation*}
\left\Vert x_{n+1}-z\right\Vert \leq \left\Vert x_{n}-z\right\Vert
\end{equation*}%
for every $n=0,1,2,\ldots .$ Then, $\left\{ P_{C}x_{n}\right\} $ converges
strongly to some $u\in C.$
\end{lemma}

\begin{lemma}
\cite{tato} Let $C$ be a nonempty closed convex subset of a real Hilbert
space $H$ and let $A$ be an $\alpha $-inverse strongly monotone mapping of $C
$ into $H.$ Then, the solution of $VI\left( C,A\right) $, $\Omega ,$ is
nonempty.
\end{lemma}

For a set-valued mapping $S:H\rightarrow 2^{H}$, if the inequality%
\begin{equation*}
\left\langle f-g,u-v\right\rangle \geq 0
\end{equation*}%
holds for all $u,v\in C,f\in Su,g\in Sv,$ then $S$ is called monotone
mapping. A monotone mapping $S:H\rightarrow 2^{H}$ is maximal if the graph $%
G\left( S\right) $ of $S$ is not properly contained in the graph of any
other monotone mappings. It is known that a monotone mapping $S$ is maximal
if and only if, for $\left( u,f\right) \in H\times H,$ $\left\langle
u-v,f-w\right\rangle \geq 0$ for every $\left( v,w\right) \in G\left(
S\right) $ implies $f\in Su.$ Let $A$ be an inverse strongly monotone
mapping of $C$ into $H,$ let $N_{C}v$ be the normal cone to $C$ at $v\in C,$
i.e.,%
\begin{equation*}
N_{C}v=\left\{ w\in H:\left\langle v-u,w\right\rangle \geq 0,\forall u\in
C\right\} ,
\end{equation*}%
and define%
\begin{equation*}
Sv=\left\{ 
\begin{array}{cc}
Av+N_{C}v & v\in C \\ 
\emptyset  & v\notin C.%
\end{array}%
\right. 
\end{equation*}%
Then, $S$ is maximal monotone and $0\in Sv$ if and only if $v\in \Omega .$

\begin{lemma}
\label{b}\cite{kirk} Let $C$ be a nonempty closed convex subset of a real
Hilbert space $H,$ and $T$ be a nonexpansive self-mapping on $C.$ If $%
F\left( T\right) \neq \emptyset ,$ then $I-T$ is demiclosed; that is
whenever $\left\{ x_{n}\right\} $ is a sequence in $C$ weakly converging to
some $x\in C$ and the sequence $\left\{ \left( I-T\right) x_{n}\right\} $
strongly converges to some $y$, it follows that $\left( I-T\right) x=y.$
Here, $I$ is the identity operator of $H.$
\end{lemma}

\begin{lemma}
\label{a}\cite{schu} Let $H$ be a real Hilbert space, let $\left\{ \alpha
_{n}\right\} $ be a sequence of real numbers such that $0<a\leq \alpha
_{n}\leq b<1$ for all $n=0,1,2,\ldots ,$ and let $\left\{ x_{n}\right\} $
and $\left\{ y_{n}\right\} $ be sequences of $H$ such that%
\begin{equation*}
\limsup_{n\rightarrow \infty }\left\Vert x_{n}\right\Vert \leq c,\text{ }%
\limsup_{n\rightarrow \infty }\left\Vert y_{n}\right\Vert \leq c\text{ and }%
\lim_{n\rightarrow \infty }\left\Vert \alpha _{n}x_{n}+\left( 1-\alpha
_{n}\right) y_{n}\right\Vert =c,
\end{equation*}%
\ for some $c>0.$ Then,%
\begin{equation*}
\lim_{n\rightarrow \infty }\left\Vert x_{n}-y_{n}\right\Vert =0.
\end{equation*}
\end{lemma}

\begin{lemma}
\label{Y}\cite{xu} Assume that $\left\{ x_{n}\right\} $ is a sequence of
nonnegative real numbers satisfying the conditions%
\begin{equation*}
x_{n+1}\leq \left( 1-\alpha _{n}\right) x_{n}+\alpha _{n}\beta _{n},\text{ }%
\forall n\geq 0
\end{equation*}%
where $\left\{ \alpha _{n}\right\} $ \ and $\left\{ \beta _{n}\right\} $ are
sequences of real numbers such that%
\begin{eqnarray*}
&\text{(i) }&\left\{ \alpha _{n}\right\} \subset \left[ 0,1\right] \text{
and }\tsum_{n=0}^{\infty }\alpha _{n}=\infty ,\text{ or equivalently }%
\tprod_{n=0}^{\infty }\left( 1-\alpha _{n}\right) =0\text{, \ \ \ \ } \\
&\text{(ii)}&\limsup_{n\rightarrow \infty }\beta _{n}\leq 0\text{, or }%
\tsum_{n}\alpha _{n}\beta _{n}<\infty \text{.}
\end{eqnarray*}%
\textit{Then,} $\lim_{n\rightarrow \infty }x_{n}=0.$
\end{lemma}

Concerning the mapping $W_{n}$ defined by (\ref{11}), we have the following
lemmas in a real Hilbert space which can be obtained from Shimoji and
Takahashi \cite{shi}.

\begin{lemma}
\label{AA}\cite{shi} Let $C$ be a nonempty closed and convex subset of a
real Hilbert space $H$. Let $\left\{ T_{n}\right\} $ be an infinite family
of nonexpansive mappings on $C$ such that $\tbigcap_{n=1}^{\infty }F\left(
T_{n}\right) \ $is nonempty, and let $\mu _{1},\mu _{2},\ldots $ be real
numbers such that $0\leq \mu _{n}\leq 1$ for all $n\in 
\mathbb{N}
$. Then, for every $x\in C$ and $k\in 
\mathbb{N}
$, the limit $\lim_{n\rightarrow \infty }U_{n,k}x$\ exists.
\end{lemma}

By using the Lemma \ref{AA}, one can define the mapping $W$ on $C$ as
follows:%
\begin{equation*}
Wx=\lim_{n\rightarrow \infty }W_{n}x=\lim_{n\rightarrow \infty }U_{n,1}x,%
\text{ }\forall x\in H.
\end{equation*}%
Such a $W$ is called the $W$-mapping generated by $T_{1},T_{2},\ldots $ and $%
\mu _{1},\mu _{2},\ldots $. Throughout this paper, we assume that $0<\mu
_{n}\leq b<1$ for $n\geq 0$.

\begin{lemma}
\label{e3}\cite{shi} Let $C$ be a nonempty closed and convex subset of a
real Hilbert space $H$. Let $\left\{ T_{n}\right\} $ be an infinite family
of nonexpansive mappings on $C$ such that $\tbigcap_{n=1}^{\infty }F\left(
T_{n}\right) \ $is nonempty, and let $\mu _{1},\mu _{2},\ldots $ be real
numbers such that $0\leq \mu _{n}\leq 1$ for $n\geq 0$. Then, $F\left(
W\right) =\tbigcap_{n=1}^{\infty }F\left( T_{n}\right) $.
\end{lemma}

\section{Main result}

Now, we are in a position to state and prove the main result in this paper.

\begin{theorem}
\label{1*}Let $C$ be a nonempty closed convex subset of a real Hilbert space 
$H$, let $A:C\rightarrow H$ be an $\alpha $-inverse strongly monotone
mapping and let $\left\{ T_{n}\right\} $ be an infinite family of
nonexpansive self-mappings on $C$ such that $\tciFourier
:=\tbigcap_{n=0}^{\infty }F\left( T_{n}\right) \cap \Omega \neq \emptyset .$
Let $\left\{ x_{n}\right\} $ be a sequence defined by (\ref{12}), where $%
\{\lambda _{n}\}\subset \lbrack a,b]$ for some $a,b\in (0,2\alpha )$ and $%
\left\{ \alpha _{n}\right\} \subset \left[ c,d\right] $ for some $c,d\in
\left( 0,1\right) $. Then, the sequence $\left\{ x_{n}\right\} $ converges
strongly to a point $z\in \tciFourier $ where $z$ is the unique solution of
the variational inequality (\ref{25}).
\end{theorem}

\begin{proof}
We devide our proof into five steps.

\textbf{Step 1. }First, we show that $\left\{ x_{n}\right\} $ is a bounded
sequence. Let $t_{n}=P_{C}\left( I-\lambda _{n}A\right) x_{n}$ and $z\in
\tciFourier $. Then, we have%
\begin{eqnarray}
\left\Vert t_{n}-z\right\Vert ^{2} &=&\left\Vert P_{C}\left( I-\lambda
_{n}A\right) x_{n}-z\right\Vert ^{2}  \notag \\
&\leq &\left\Vert \left( I-\lambda _{n}A\right) x_{n}-\left( I-\lambda
_{n}A\right) z\right\Vert ^{2}  \notag \\
&=&\left\Vert x_{n}-z-\lambda _{n}\left( Ax_{n}-Az\right) \right\Vert ^{2} 
\notag \\
&\leq &\left\Vert x_{n}-z\right\Vert ^{2}-2\lambda _{n}\left\langle
x_{n}-z,Ax_{n}-Az\right\rangle +\lambda _{n}^{2}\left\Vert
Ax_{n}-Az\right\Vert ^{2}  \notag \\
&\leq &\left\Vert x_{n}-z\right\Vert ^{2}+\lambda _{n}\left( \lambda
_{n}-2\alpha \right) \left\Vert Ax_{n}-Az\right\Vert ^{2}  \notag \\
&\leq &\left\Vert x_{n}-z\right\Vert ^{2}  \label{1}
\end{eqnarray}%
and from (\ref{1}) we get%
\begin{eqnarray}
\left\Vert x_{n+1}-z\right\Vert ^{2} &=&\left\Vert W_{n}P_{C}\left(
I-\lambda _{n}A\right) y_{n}-z\right\Vert ^{2}  \notag \\
&=&\left\Vert W_{n}P_{C}\left( I-\lambda _{n}A\right) y_{n}-W_{n}P_{C}\left(
I-\lambda _{n}A\right) z\right\Vert ^{2}  \notag \\
&\leq &\left\Vert y_{n}-z\right\Vert ^{2}  \notag \\
&=&\left\Vert \left( 1-\alpha _{n}\right) \left( x_{n}-z\right) +\alpha
_{n}\left( W_{n}t_{n}-z\right) \right\Vert ^{2}  \notag \\
&\leq &\left( 1-\alpha _{n}\right) \left\Vert x_{n}-z\right\Vert ^{2}+\alpha
_{n}\left\Vert W_{n}t_{n}-z\right\Vert ^{2}  \notag \\
&\leq &\left( 1-\alpha _{n}\right) \left\Vert x_{n}-z\right\Vert ^{2}+\alpha
_{n}\left\Vert t_{n}-z\right\Vert ^{2}  \notag \\
&\leq &\left( 1-\alpha _{n}\right) \left\Vert x_{n}-z\right\Vert ^{2}  \notag
\\
&&+\alpha _{n}\left[ \left\Vert x_{n}-z\right\Vert ^{2}+\lambda _{n}\left(
\lambda _{n}-2\alpha \right) \left\Vert Ax_{n}-Az\right\Vert ^{2}\right]  
\notag \\
&=&\left\Vert x_{n}-z\right\Vert ^{2}+\alpha _{n}\lambda _{n}\left( \lambda
_{n}-2\alpha \right) \left\Vert Ax_{n}-Az\right\Vert ^{2}  \notag \\
&\leq &\left\Vert x_{n}-z\right\Vert ^{2}+da\left( b-2\alpha \right)
\left\Vert Ax_{n}-Az\right\Vert ^{2}  \notag \\
&\leq &\left\Vert x_{n}-z\right\Vert ^{2}.  \label{2*}
\end{eqnarray}%
Therefore, the limit $\lim_{n\rightarrow \infty }\left\Vert
x_{n}-z\right\Vert $ exists and $Ax_{n}-Az\rightarrow 0.$ Hence, $\left\{
x_{n}\right\} $ is bounded and so are $\left\{ t_{n}\right\} $ and $\left\{
W_{n}t_{n}\right\} $.

\textbf{Step 2.} We will show that $\lim_{n\rightarrow \infty }\left\Vert
x_{n}-y_{n}\right\Vert =0.$ Before that, we shall show that $%
\lim_{n\rightarrow \infty }\left\Vert W_{n}t_{n}-x_{n}\right\Vert =0$. From
Step 1, we know that $\lim_{n\rightarrow \infty }\left\Vert
x_{n}-z\right\Vert $ exists for all $z\in \tciFourier $.\ Let $%
\lim_{n\rightarrow \infty }\left\Vert x_{n}-z\right\Vert =c.$\ From (\ref{2*}%
), since%
\begin{equation*}
\left\Vert x_{n+1}-z\right\Vert \leq \left\Vert y_{n}-z\right\Vert \leq
\left\Vert x_{n}-z\right\Vert ,
\end{equation*}%
we get%
\begin{equation}
\lim_{n\rightarrow \infty }\left\Vert y_{n}-z\right\Vert =c.  \label{*1}
\end{equation}%
On the other hand, since%
\begin{equation*}
\left\Vert W_{n}t_{n}-z\right\Vert \leq \left\Vert t_{n}-z\right\Vert \leq
\left\Vert x_{n}-z\right\Vert ,
\end{equation*}%
we have%
\begin{equation}
\limsup_{n\rightarrow \infty }\left\Vert W_{n}t_{n}-z\right\Vert \leq c.
\label{*2}
\end{equation}%
Also, we know that%
\begin{equation}
\limsup_{n\rightarrow \infty }\left\Vert x_{n}-z\right\Vert \leq c
\label{*3}
\end{equation}%
and%
\begin{equation}
\lim_{n\rightarrow \infty }\left\Vert y_{n}-z\right\Vert =\lim_{n\rightarrow
\infty }\left\Vert \left( 1-\alpha _{n}\right) \left( x_{n}-z\right) +\alpha
_{n}\left( W_{n}t_{n}-z\right) \right\Vert =c.  \label{*4}
\end{equation}%
Hence, from (\ref{*2}), (\ref{*3}), (\ref{*4}), and Lemma \ref{a} , we get
that%
\begin{equation}
\lim_{n\rightarrow \infty }\left\Vert x_{n}-W_{n}t_{n}\right\Vert =0.
\label{*5}
\end{equation}%
We have also%
\begin{equation*}
\left\Vert x_{n}-y_{n}\right\Vert =\alpha _{n}\left\Vert
W_{n}t_{n}-x_{n}\right\Vert .
\end{equation*}%
So, from (\ref{*5}) we obtain that%
\begin{equation}
\lim_{n\rightarrow \infty }\left\Vert x_{n}-y_{n}\right\Vert =0.  \label{5.5}
\end{equation}%
Since $A$ is Lipschitz continuous, we have $Ax_{n}-Ay_{n}\rightarrow 0.$

\textbf{Step 3.} Now, we show that $\lim_{n\rightarrow \infty }\left\Vert
Wx_{n}-x_{n}\right\Vert =0.$ Using the properties of the metric projection,
since%
\begin{eqnarray*}
\left\Vert t_{n}-z\right\Vert ^{2} &=&\left\Vert P_{C}\left( I-\lambda
_{n}A\right) x_{n}-P_{C}\left( I-\lambda _{n}A\right) z\right\Vert ^{2} \\
&\leq &\left\langle t_{n}-z,\left( I-\lambda _{n}A\right) x_{n}-\left(
I-\lambda _{n}A\right) z\right\rangle  \\
&=&\frac{1}{2}\left[ \left\Vert t_{n}-z\right\Vert ^{2}+\left\Vert \left(
I-\lambda _{n}A\right) x_{n}-\left( I-\lambda _{n}A\right) z\right\Vert
^{2}\right.  \\
&&\left. -\left\Vert t_{n}-z-\left[ \left( I-\lambda _{n}A\right)
x_{n}-\left( I-\lambda _{n}A\right) z\right] \right\Vert ^{2}\right]  \\
&\leq &\frac{1}{2}\left[ \left\Vert t_{n}-z\right\Vert ^{2}+\left\Vert
x_{n}-z\right\Vert ^{2}-\left\Vert \left( t_{n}-x_{n}\right) -\lambda
_{n}\left( Ax_{n}-Az\right) \right\Vert ^{2}\right]  \\
&=&\frac{1}{2}\left[ \left\Vert t_{n}-z\right\Vert ^{2}+\left\Vert
x_{n}-z\right\Vert ^{2}-\left\Vert t_{n}-x_{n}\right\Vert ^{2}\right.  \\
&&\left. -2\lambda _{n}\left\langle t_{n}-x_{n},Ax_{n}-Az\right\rangle
-\lambda _{n}^{2}\left\Vert Ax_{n}-Az\right\Vert ^{2}\right] ,
\end{eqnarray*}%
it follows that%
\begin{eqnarray}
\left\Vert t_{n}-z\right\Vert ^{2} &\leq &\left\Vert x_{n}-z\right\Vert
^{2}-\left\Vert t_{n}-x_{n}\right\Vert ^{2}  \notag \\
&&+2\lambda _{n}\left\langle t_{n}-x_{n},Ax_{n}-Az\right\rangle -\lambda
_{n}^{2}\left\Vert Ax_{n}-Az\right\Vert ^{2}.  \label{8}
\end{eqnarray}%
So, by using the inequality (\ref{8}) and (\ref{2*}),\ we get%
\begin{eqnarray*}
\left\Vert x_{n+1}-z\right\Vert ^{2} &\leq &\left( 1-\alpha _{n}\right)
\left\Vert x_{n}-z\right\Vert ^{2}+\alpha _{n}\left\Vert t_{n}-z\right\Vert
^{2} \\
&\leq &\left\Vert x_{n}-z\right\Vert ^{2}-\alpha _{n}\left\Vert
t_{n}-x_{n}\right\Vert ^{2} \\
&&+2\lambda _{n}\alpha _{n}\left\langle t_{n}-x_{n},Ax_{n}-Az\right\rangle
-\lambda _{n}^{2}\alpha _{n}\left\Vert Ax_{n}-Az\right\Vert ^{2} \\
&\leq &\left\Vert x_{n}-z\right\Vert ^{2}-d\left\Vert t_{n}-x_{n}\right\Vert
^{2} \\
&&+2\lambda _{n}\alpha _{n}\left\langle t_{n}-x_{n},Ax_{n}-Az\right\rangle
-\lambda _{n}^{2}\alpha _{n}\left\Vert Ax_{n}-Az\right\Vert ^{2}.
\end{eqnarray*}%
Since $\lim_{n\rightarrow \infty }\left\Vert x_{n+1}-z\right\Vert
=\lim_{n\rightarrow \infty }\left\Vert x_{n}-z\right\Vert $ and $%
Ax_{n}-Az\rightarrow 0,$\ we obtain%
\begin{equation}
\lim_{n\rightarrow \infty }\left\Vert x_{n}-t_{n}\right\Vert =0.  \label{*6}
\end{equation}%
On the other hand, we have%
\begin{eqnarray*}
\left\Vert W_{n}x_{n}-x_{n}\right\Vert  &\leq &\left\Vert
W_{n}x_{n}-W_{n}t_{n}\right\Vert +\left\Vert W_{n}t_{n}-x_{n}\right\Vert  \\
&\leq &\left\Vert x_{n}-t_{n}\right\Vert +\left\Vert
W_{n}t_{n}-x_{n}\right\Vert .
\end{eqnarray*}%
So, it follows from (\ref{*5}) and (\ref{*6}) that%
\begin{equation}
\lim_{n\rightarrow \infty }\left\Vert W_{n}x_{n}-x_{n}\right\Vert =0.
\label{20}
\end{equation}%
Hence, from (\ref{20})\ and by the same argument as in the \cite[Remark 2.2]%
{ceng}, it follows that%
\begin{equation}
\left\Vert Wx_{n}-x_{n}\right\Vert \leq \left\Vert
Wx_{n}-W_{n}x_{n}\right\Vert +\left\Vert W_{n}x_{n}-x_{n}\right\Vert
\rightarrow 0,  \label{*7}
\end{equation}%
as $n\rightarrow \infty $.

\textbf{Step 4. }Next, we show that 
\begin{equation*}
\limsup_{n\rightarrow \infty }\left[ \left\langle
W_{n}t_{n}-z,x_{n}-z\right\rangle +\left\Vert W_{n}t_{n}-z\right\Vert ^{2}%
\right] \leq 0,
\end{equation*}%
where $z\in \tciFourier $.\ But first, we need to show that the variational
inequality (\ref{25}) has unique solution. Indeed, suppose both $p\in C$ and 
$q\in C$ are solutions to (\ref{25}), then%
\begin{equation}
\left\langle Ap,p-q\right\rangle \leq 0  \label{26}
\end{equation}%
and%
\begin{equation}
\left\langle Aq,q-p\right\rangle \leq 0.  \label{27}
\end{equation}%
Combining (\ref{26}) and (\ref{27}), we get%
\begin{equation}
\left\langle Aq-Ap,q-p\right\rangle \leq 0.  \label{28}
\end{equation}%
Since the mapping $A$ is an inverse strongly monotone mapping, (\ref{28})
implies $p=q.$ So, the uniqueness of the solution of the variational
inequality (\ref{25}) is proved. Next, we need to show that $\left\{
x_{n}\right\} $ converges weakly to an element of $\tciFourier $. Since $%
\left\{ x_{n}\right\} $ and $\left\{ W_{n}t_{n}\right\} $ are bounded
sequences, there exist subsequences $\left\{ x_{n_{i}}\right\} $ of $\left\{
x_{n}\right\} $ and $\left\{ W_{n}t_{n_{i}}\right\} $ of $\left\{
W_{n}t_{n}\right\} $ such that%
\begin{eqnarray}
&&\limsup_{n\rightarrow \infty }\left[ \left\langle
W_{n}t_{n}-z,x_{n}-z\right\rangle +\left\Vert W_{n}t_{n}-z\right\Vert ^{2}%
\right]   \notag \\
&=&\limsup_{i\rightarrow \infty }\left[ \left\langle
W_{n}t_{n_{i}}-z,x_{n_{i}}-z\right\rangle +\left\Vert
W_{n}t_{n_{i}}-z\right\Vert ^{2}\right] .  \label{21}
\end{eqnarray}%
Without loss of generality, we may further assume that $x_{n_{i}}%
\rightharpoonup p.$ From (\ref{*5}), we have $W_{n}t_{n_{i}}\rightharpoonup p
$. Hence, (\ref{21}) reduces to%
\begin{equation*}
\limsup_{n\rightarrow \infty }\left[ \left\langle
W_{n}t_{n}-z,x_{n}-z\right\rangle +\left\Vert W_{n}t_{n}-z\right\Vert ^{2}%
\right] =2\left\Vert p-z\right\Vert ^{2}
\end{equation*}%
Now, it is sufficient to show that $p$ belongs to $\tciFourier ,$ i.e., $p=z$%
. First, we show that $p\in \Omega .$ Let%
\begin{equation*}
Sv=\left\{ 
\begin{array}{ll}
Av+N_{C}v & ,\text{ }v\in C, \\ 
\emptyset  & ,\text{ }v\notin C.%
\end{array}%
\right. 
\end{equation*}%
Then, $S$ is maximal monotone mapping. Let $\left( v,w\right) \in G\left(
S\right) .$ Since $w-Av\in N_{C}v$ and $t_{n}\in C,$ we get%
\begin{equation}
\left\langle v-t_{n},w-Av\right\rangle \geq 0.  \label{10}
\end{equation}%
On the other hand, from the definiton of $t_{n},$ we have that%
\begin{equation*}
\left\langle x_{n}-\lambda _{n}Ax_{n}-t_{n},t_{n}-v\right\rangle \geq 0
\end{equation*}%
and hence,%
\begin{equation*}
\left\langle v-t_{n},\frac{t_{n}-x_{n}}{\lambda _{n}}+Ax_{n}\right\rangle
\geq 0.
\end{equation*}%
Therefore, using (\ref{10}), we get%
\begin{eqnarray*}
\left\langle v-t_{n_{i}},w\right\rangle  &\geq &\left\langle
v-t_{n_{i}},Av\right\rangle  \\
&\geq &\left\langle v-t_{n_{i}},Av\right\rangle -\left\langle v-t_{n_{i}},%
\frac{t_{n_{i}}-x_{n_{i}}}{\lambda _{n_{i}}}+Ax_{n_{i}}\right\rangle  \\
&=&\left\langle v-t_{n_{i}},Av-Ax_{n_{i}}-\frac{t_{n_{i}}-x_{n_{i}}}{\lambda
_{n_{i}}}\right\rangle  \\
&=&\left\langle v-t_{n_{i}},Av-At_{n_{i}}\right\rangle +\left\langle
v-t_{n_{i}},At_{n_{i}}-Ax_{n_{i}}\right\rangle  \\
&&-\left\langle v-t_{n_{i}},\frac{t_{n_{i}}-x_{n_{i}}}{\lambda _{n_{i}}}%
\right\rangle  \\
&\geq &\left\langle v-t_{n_{i}},At_{n_{i}}-Ax_{n_{i}}\right\rangle
-\left\langle v-t_{n_{i}},\frac{t_{n_{i}}-x_{n_{i}}}{\lambda _{n_{i}}}%
\right\rangle .
\end{eqnarray*}%
Hence, for $i\rightarrow \infty ,$ we have%
\begin{equation*}
\left\langle v-p,w\right\rangle \geq 0.
\end{equation*}%
Since $S$ is maximal monotone, we have $p\in S^{-1}0$ and hence $p\in \Omega
.$ Next, we show that $p\in F\left( W\right) .$ From (\ref{*7}), Lemma \ref%
{b} and by using $x_{n_{i}}\rightharpoonup p$, we have that $p\in F\left(
W\right) .$ So, from Lemma \ref{e3}, we get $p\in \tciFourier $.

Also, Opial's condition guarantee that the weakly subsequential limit of $%
\left\{ x_{n}\right\} $ is unique. Hence, this implies that $%
x_{n}\rightharpoonup p\in \tciFourier .$ From the uniqueness of the solution
of the variational inequality, we obtain $p=z\in \tciFourier $. So, the
desired conclusion%
\begin{equation*}
\limsup_{n\rightarrow \infty }\left[ \left\langle
W_{n}t_{n}-z,x_{n}-z\right\rangle +\left\Vert W_{n}t_{n}-z\right\Vert ^{2}%
\right] \leq 0
\end{equation*}%
is obtained.

Furthermore, $p=\lim_{n\rightarrow \infty }P_{\tciFourier }x_{n}.$ Indeed,
since $p\in \tciFourier ,$ we have%
\begin{equation*}
\left\langle p-P_{\tciFourier }x_{n},P_{\tciFourier
}x_{n}-x_{n}\right\rangle \geq 0.
\end{equation*}%
By Lemma \ref{c}, $\left\{ P_{\tciFourier }x_{n}\right\} $ converges
strongly to $u_{0}\in \tciFourier .$ Then, we get%
\begin{equation*}
\left\langle p-u_{0},u_{0}-p\right\rangle \geq 0,
\end{equation*}%
and hence $p=u_{0}.$

\textbf{Step 5.\ }\ Let\textbf{\ }$z\in \tciFourier .$ Then, we have%
\begin{eqnarray*}
\left\Vert x_{n+1}-z\right\Vert ^{2} &=&\left\Vert W_{n}P_{C}\left(
I-\lambda _{n}A\right) y_{n}-z\right\Vert ^{2} \\
&=&\left\Vert W_{n}P_{C}\left( I-\lambda _{n}A\right) y_{n}-W_{n}P_{C}\left(
I-\lambda _{n}A\right) z\right\Vert ^{2} \\
&\leq &\left\Vert y_{n}-z\right\Vert ^{2}=\left\langle
y_{n}-z,y_{n}-z\right\rangle \\
&=&\left\langle \left( 1-\alpha _{n}\right) \left( x_{n}-z\right) +\alpha
_{n}\left( W_{n}t_{n}-z\right) ,y_{n}-z\right\rangle \\
&=&\left( 1-\alpha _{n}\right) \left\langle x_{n}-z,y_{n}-z\right\rangle
+\alpha _{n}\left\langle W_{n}t_{n}-z,y_{n}-z\right\rangle \\
&\leq &\left( 1-\alpha _{n}\right) \left\Vert x_{n}-z\right\Vert ^{2}+\alpha
_{n}\left\langle W_{n}t_{n}-z,y_{n}-z\right\rangle \\
&=&\left( 1-\alpha _{n}\right) \left\Vert x_{n}-z\right\Vert ^{2}+\alpha
_{n}^{2}\left\langle W_{n}t_{n}-z,x_{n}-z\right\rangle \\
&&+\alpha _{n}\left( 1-\alpha _{n}\right) \left\langle
W_{n}t_{n}-z,W_{n}t_{n}-z\right\rangle \\
&=&\left( 1-\alpha _{n}\right) \left\Vert x_{n}-z\right\Vert ^{2}+\alpha
_{n}\beta _{n}
\end{eqnarray*}
where $\beta _{n}=\alpha _{n}\left\langle W_{n}t_{n}-z,x_{n}-z\right\rangle
+\left( 1-\alpha _{n}\right) \left\Vert W_{n}t_{n}-z\right\Vert ^{2}$. Thus
an application of Lemma \ref{Y} combined with Step 4 yields that the
sequence $\left\{ x_{n}\right\} $ defined by (\ref{12}) converges strongly
to the unique element $z\in \tciFourier .$
\end{proof}

\begin{corollary}
Let $C$ be a nonempty closed convex subset of a real Hilbert space $H,$ let $%
A:C\rightarrow H$ be an $\alpha $-inverse strongly monotone mapping and let $%
T$ be a nonexpansive self-mappings on $C$ such that $F\left( T\right) \cap
\Omega \neq \emptyset .$ Let $\left\{ x_{n}\right\} $ be a sequence defined
by%
\begin{equation*}
\left\{ 
\begin{array}{l}
x_{0}=x\in C \\ 
x_{n+1}=TP_{C}\left( I-\lambda _{n}A\right) y_{n} \\ 
y_{n}=\left( 1-\alpha _{n}\right) x_{n}+\alpha _{n}TP_{C}\left( I-\lambda
_{n}A\right) x_{n},\forall n\geq 0,%
\end{array}%
\right. 
\end{equation*}%
where $\{\lambda _{n}\}\subset \lbrack a,b]$ for some $a,b\in (0,2\alpha )$
and $\left\{ \alpha _{n}\right\} \subset \left[ c,d\right] $ for some $%
c,d\in \left( 0,1\right) $. Then, the sequence $\left\{ x_{n}\right\} $
converges strongly to a point $z\in F\left( T\right) \cap \Omega $ where $z$
is the unique solution of the variational inequality (\ref{25}).
\end{corollary}

\section{Applications}

In the first section, we state that the convex minimization problem is one
of the application area of the variational inequality problems and the fixed
point problems. One of the relationships between a convex minimization
problem and a variational inequality problem is as follows: Let $f$ be a
convex differentiable function on a nonempty closed convex subset $C$ of a
real Hilbert space $H$ and $\limfunc{Argmin}_{x\in C}f\left( x\right) $ be
the set of minimizers of $f$ relative to the set $C$. Then, it is known that
element $x^{\ast }\in C$ is a minimizer of $f\left( x\right) $ if and only
if $x^{\ast }$ satisfies the variational inequality (\ref{25}). On the other
hand, iterative processes are often used to minimize a convex differentiable
function. Also, it is stated in Remark \ref{y} that every $L$-Lipschitzian
mapping is $2/L$-inverse strongly monotone mapping. Therefore, we can give
the following strong convergence theorem.

\begin{theorem}
Let $C$ be a nonempty closed convex subset of a real Hilbert space $H$. Let $%
f$ be a convex differentiable function on an open set $D$ containing the set 
$C$ and let $\left\{ T_{n}\right\} $ be an infinite family of nonexpansive
self mappings on $C$ such that $\mathcal{G=}\bigcap_{n=0}^{\infty }F\left(
T_{n}\right) \cap \limfunc{Argmin}_{x\in C}f\left( x\right) \neq \emptyset $%
. Suppose that the gradient vector of $f$, $\nabla f,$ is a $L$-Lipschitz
continuous operator on $D$. For an arbitrarily initial value $x_{0}\in C,$
let $\left\{ x_{n}\right\} $ be a sequence in $C$ defined by%
\begin{equation*}
\left\{ 
\begin{array}{l}
x_{n+1}=W_{n}P_{C}\left( I-\lambda _{n}\nabla f\right) y_{n} \\ 
y_{n}=\left( 1-\alpha _{n}\right) x_{n}+\alpha _{n}W_{n}P_{C}\left(
I-\lambda _{n}\nabla f\right) x_{n},\forall n\geq 0,%
\end{array}%
\right. 
\end{equation*}%
where $W_{n}$ is a mapping defined by (\ref{11}), $\{\lambda _{n}\}\subset
\lbrack a,b]$ for some $a,b\in (0,4/L)$ and $\left\{ \alpha _{n}\right\}
\subset \left[ c,d\right] $ for some $c,d\in \left( 0,1\right) $. Then the
sequence $\left\{ x_{n}\right\} $ converges strongly to an element of $%
\mathcal{G}$.
\end{theorem}

\begin{proof}
Considering the Remark \ref{y}, as in the proof of Theorem \ref{1*}, if we
take $A=\nabla f$, then we obtain the desired conclusion.
\end{proof}

Next, we give another theorem for a pair of nonexpansive mapping and
strictly pseudocontractive mapping. A mapping $S:C\rightarrow C$ is called $k
$-strictly pseudocontractive mapping if there exists $k$ with $0\leq k<1$
such that%
\begin{equation*}
\left\Vert Sx-Sy\right\Vert ^{2}\leq \left\Vert x-y\right\Vert
^{2}+k\left\Vert \left( I-S\right) x-\left( I-S\right) y\right\Vert ^{2}
\end{equation*}%
for all $x,y\in C.$ Let $A=I-S.$ Then, it is known that the mapping $A$ is
inverse strongly monotone mapping with $\left( 1-k\right) /2$, i.e.,%
\begin{equation*}
\left\langle Ax-Ay,x-y\right\rangle \geq \frac{1-k}{2}\left\Vert
Ax-Ay\right\Vert ^{2}.
\end{equation*}

\begin{theorem}
Let $C$ be a nonempty closed convex subset of a real Hilbert space $H.$ Let $%
\left\{ T_{n}\right\} $ be an infinite family of nonexpansive self mappings
on $C$ and $S:C\rightarrow C$ be a $k$-strictly pseudocontractive mapping
such that $\mathcal{H}=\bigcap_{n=0}^{\infty }F\left( T_{n}\right) \cap
F\left( S\right) \neq \emptyset .$ For an arbitrarily initial value $%
x_{0}\in C,$ let $\left\{ x_{n}\right\} $ be a sequence defined by%
\begin{equation*}
\left\{ 
\begin{array}{l}
x_{n+1}=W_{n}\left( \left( I-\lambda _{n}\right) y_{n}+\lambda
_{n}Sy_{n}\right)  \\ 
y_{n}=\left( 1-\alpha _{n}\right) x_{n}+\alpha _{n}W_{n}\left( \left(
I-\lambda _{n}\right) x_{n}+\lambda _{n}Sx_{n}\right) ,\forall n\geq 0,%
\end{array}%
\right. 
\end{equation*}%
where $\{\lambda _{n}\}\subset \lbrack a,b]$ for some $a,b\in (0,1-k)$ and $%
\left\{ \alpha _{n}\right\} \subset \left[ c,d\right] $ for some $c,d\in
\left( 0,1\right) $. Then, the sequence $\left\{ x_{n}\right\} $ converges
weakly to a point $p\in \mathcal{H}$.
\end{theorem}

\begin{proof}
Let $A=I-S.$ Then, we know that $A$ is inverse strongly monotone mapping.
Also, it is clear that $F\left( S\right) =VI\left( C,A\right) .$ Since, $A$
is a mapping from $C$ into itself, we get%
\begin{equation*}
\left( I-\lambda _{n}\right) x_{n}+\lambda _{n}Sx_{n}=x_{n}-\lambda
_{n}\left( I-S\right) x_{n}=P_{C}\left( I-\lambda _{n}A\right) x_{n}.
\end{equation*}%
So, from Theorem \ref{1*}, we obtain the desired conclusion.
\end{proof}

\end{document}